\documentclass[12pt]{article}
\usepackage[english]{babel}
\usepackage[OT2,T1]{fontenc}
\usepackage{tikz}
\usepackage{amscd}
\usepackage[all,cmtip]{xy}
\usepackage{amssymb,amsmath,mathrsfs,amsthm}
\usetikzlibrary{decorations.pathreplacing}
\DeclareSymbolFont{cyrletters}{OT2}{wncyr}{m}{n}
\DeclareMathSymbol{\Sha}{\mathalpha}{cyrletters}{"58}

\newcommand\wt[1]{\widetilde{#1}}

\newtheorem{thm}{Theorem}[section]
\newtheorem{defn}[thm]{Definition}

\newtheorem{lem}[thm]{Lemma}
\newtheorem{prop}[thm]{Proposition}

\newtheorem{rem}[thm]{Remark}

\newcommand{\p}{\mathbf{P}}

\newcommand{\Q}{\mathbf{Q}}
\newcommand{\A}{\mathbf{A}}

\newcommand{\Z}{\mathbf{Z}}

\newcommand{\KK}{\mathcal{K}}
\newcommand{\F}{\mathbf{F}}

\newcommand{\Km}{\operatorname{Km}}

\title{On $p$-adic density of rational points on K3 surfaces}
\author{Ren\'e Pannekoek}
\date{\today}

\begin{document}

\maketitle
\abstract{We show that, for every prime number $p$, there exist infinitely many K3 surfaces over $\Q$ whose rational points lie dense in the space of its $p$-adic points. We also show that there exists a K3 surface over $\Q$ whose rational points lie dense in the space of its $p$-adic points for all prime numbers $p$ with $p \equiv 3 \pmod{4}$ and $p>7$.}

\section{Introduction}

In an unpublished preprint \cite{unpub}, Sir Peter Swinnerton-Dyer gave three non-singular diagonal quartic surfaces over $\Q$ together with a proof that their rational points lie dense in the space of $2$-adic points. To the author's best knowledge, this is the first instance of a proof of $p$-adic density of rational points on any K3 surface over $\Q$, for any prime number $p$. The goal of this article is to extend the results of Swinnerton-Dyer to all prime numbers $p$, giving for each $p$ an infinite number of K3 surfaces over $\Q$ on which the rational points form a $p$-adically dense set.

The K3 surfaces for which we will obtain $p$-adic density results will be Kummer surfaces. For an abelian variety $A$ over a field of characteristic different from $2$, let $\Km(A)$ denote the Kummer variety of $A$. It is the blow-up of the quotient $A/\left\langle -1 \right\rangle$ in the image of the $2$-torsion of $A$. When $A$ is an abelian variety of dimension $2$, the surface $\Km(A)$ is a K3 surface.

We will establish the following results.
\begin{thm}
\label{swd}
Let $p$ be a prime number. Then there exist infinitely many elliptic curves $E$ over $\Q$ such that the rational points of $\Km(E \times E)$ lie dense in the space of $p$-adic points.
\end{thm}
\begin{thm}
\label{cmtrick}
There exists an elliptic curve $E$ over $\Q$ such that the rational points of $\Km(E \times E)$ lie dense in the space of $p$-adic points for all prime numbers $p$ with $p \equiv 3 \pmod{4}$ and $p>7$.
\end{thm}

The proofs of Theorems \ref{swd} and \ref{cmtrick} are given at the end of Section \ref{main}. 

We end this section by fixing some notation. If $E$ is an elliptic curve over any field $k$, and $c \in k^{\ast}$, then by $E^c$ we denote the quadratic twist of $E$ by $c$. If $E$ is given by a Weierstrass equation of the form $y^2=f(x)$, then $E^c$ is isomorphic to the elliptic curve given by $cy^2=f(x)$. By $E_0$ we denote the complement of $E[2]$ in $E$.

\section{Elliptic curves with suitable twists}

\begin{defn}\label{geschikt}
\upshape
We will say that an elliptic curve $E$ over $\Q$ has \textit{suitable twists} with respect to a prime number $p$ if for all $d \in \Q_p^{\ast}$ there exists $c \in \Q^{\ast}$ such that $d/c \in \Q_p^{\ast 2}$ and $E^c(\Q)$ is dense in $E^c(\Q_p)$.
\end{defn}
Theorem \ref{big} will show: if the elliptic curve $E$ over $\Q$ has suitable twists with respect to $p$, and we have $X = \Km(E \times E)$, then $X(\Q)$ is dense in $X(\Q_p)$.
\begin{rem}\upshape
The condition $d/c \in \Q_p^{\ast 2}$ appearing in Definition \ref{geschikt} is equivalent to the twists $E^c$ and $E^d$, considered as elliptic curves over $\Q_p$, being isomorphic over $\Q_p$. We may thus rephrase the fact of $E$ having suitable twists with respect to $p$ as follows: for all twists $E^d$ of $E$ over $\Q_p$, there exists a twist $E^c$ of $E$ over $\Q$ which is isomorphic to $E^d$ over $\Q_p$, for which $E^c(\Q)$ is dense in $E^c(\Q_p)$.
\end{rem}

The remainder of this section is used to show that there exist many suitable elliptic curves $E$ for any prime number $p$ (Proposition \ref{oneindigveel}). 

\begin{lem}
\label{swdlemma}
Let $p>7$ be a prime number. Let $E$ be an elliptic curve over $\Q_p$ with additive reduction, and write $E^{(0)}(\Q_p)$ for the subgroup of $E(\Q_p)$ consisting of the points that have good reduction. Then $E^{(0)}(\Q_p)$ is topologically isomorphic to $\Z_p$.
\end{lem}
\begin{proof}
This result was already observed by Swinnerton-Dyer as Lemma 1 of \cite{unpub}. The result appears with proof as Theorem 1 of \cite{Rene}. For the convenience of the reader, we here reproduce the arguments from \cite{Rene}.

From the theory of elliptic curves over local fields (see Chapter 7 of \cite{Silverman}), we get an exact sequence:
\begin{equation}
\label{Elocalfields}
0 \rightarrow E^{(1)}(\Q_p) \rightarrow E^{(0)}(\Q_p) \rightarrow \widetilde{E}_{\textnormal{ns}}(\mathbf{F}_p) \rightarrow 0,
\end{equation}
where $\widetilde{E}_{\textnormal{ns}}(\mathbf{F}_p)$ denotes the group of non-singular points over $\F_p$ on a minimal Weierstrass model of $E$, the arrow $E^{(0)}(\Q_p) \rightarrow \widetilde{E}_{\textnormal{ns}}(\mathbf{F}_p)$ is the reduction map, and we write $E^{(1)}(\Q_p)$ for the kernel of the reduction map. It follows from \cite[IV.6.4(b)]{Silverman} that $E^{(1)}(\Q_p)$ is canonically isomorphic to $\Z_p$. Since $E$ has additive reduction at $p$, we have $\widetilde{E}_{\textnormal{ns}}(\mathbf{F}_p) \cong \Z/p\Z$. Hence, the short exact sequence (\ref{splitses}) reads
$$
0 \rightarrow \Z_p \rightarrow E^{(0)}(\Q_p) \rightarrow \Z/p\Z \rightarrow 0.
$$
It follows that the topological group $E^{(0)}(\Q_p)$ is isomorphic to $\Z_p \times \Z/p\Z$ if and only if it has non-trivial $p$-torsion; otherwise, it is isomorphic to $\Z_p$. 

We will show that $E^{(0)}(\Q_p)$ has no non-trivial $p$-torsion. We will make use of a ramified field extension of $\Q_p$. Assume that $E$ is given by a Weierstrass equation $y^2=x^3+ax+b$ that is minimal at $p$. Let $K = \Q_p(\pi)$, with $\pi^6=p$. We define a new curve $E'$ given by $y^2 = x^3 + a/\pi^4 x + b/\pi^6$. There is an isomorphism $\phi : E \stackrel{\sim}{\rightarrow} E'$ defined over $K$, given by $\phi(x,y) = (x/\pi^2,y/\pi^3)$. Now $\phi$ injects $E^{(0)}(\Q_p)$ into the kernel of reduction $(E')^{(1)}(K)$ of $E'$, which by \cite[IV.6.4(b)]{Silverman} is isomorphic to the ring of integers of $K$ (this uses $p > 7$), which is torsion-free. So $E^{(0)}(\Q_p)$ is torsion-free, hence topologically isomorphic to $\Z_p$.
\end{proof}

We recall that a topological group $G$ is called procyclic if, for some $g \in G$, the subgroup generated by $g$ is dense in $G$. This element $g$ is called a topological generator of $G$.

\begin{lem}
\label{existenceresult}
Let $p$ be a prime. There exist infinitely many elliptic curves $E$ over $\Q$ such that, for all $d \in \Q_p^{\ast}$, the topological group $E^d(\Q_p)$ is procyclic.
\end{lem}
\begin{proof}
First assume $p > 7$. Choose the elliptic curve $E$ over $\Q$ such that its Kodaira reduction type at $p$ is in the set $\KK = \{ \textnormal{II,III,IV,II${}^{\ast}$,III${}^{\ast}$, IV${}^{\ast}$} \}$. It is obvious that there are infinitely many such $E$ for each $p$ (e.g., see the table in \cite[C.15]{Silverman}). We will show that $E$ satisfies the conclusion of the lemma. Note that $E$ has additive reduction at $p$. The class of elliptic curves over $\Q$ with reduction type at $p$ contained in the set $\KK$ is stable under taking quadratic twists. Therefore, we may reduce to showing that $E(\Q_p)$ is procyclic. 

From \cite[C.15]{Silverman} we have that $E(\Q_p)$ fits inside a short exact sequence:
\begin{equation}
\label{splitses}
0 \rightarrow E^{(0)}(\Q_p) \rightarrow E(\Q_p) \rightarrow \Z/m \Z \rightarrow 0,
\end{equation}
where $m \in \{1,2,3\}$. By Lemma \ref{swdlemma}, the topological group $E^{(0)}(\Q_p)$ is isomorphic to $\Z_p$. Since $p \nmid m$ by assumption on $p$, it follows that (\ref{splitses}) splits, and that $E(\Q_p)$ is topologically isomorphic to $\Z_p \times \Z/m\Z$. This proves the lemma for $p > 7$.

For $p \leq 7$, we give examples. Let $E$ be an elliptic curve over $\Q_p$ given by $y^2=x^3+ax+b$. Then $E(\Q_p)$ is procyclic in each of the cases (i) $p=2$, $v_2(a) \geq 1$, $v_2(b) = 1$; (ii) $p=3$, $v_3(a)= 1$, $v_3(b) >1$; (iii) $p=5$, $v_5(a)\geq 1$, $v_5(b) =1$, $a \not\equiv \pm 10 \pmod{25}$; (iv) $p=7$, $v_7(a)\geq 1$, $v_7(b) =1$, $b \not\equiv \pm 14 \pmod{49}$. One simply shows this by looking at the various division polynomials associated to $E$, ruling out any unwanted torsion. This completes the proof for $p \leq 7$.
\end{proof}

We will now show that the property isolated in the preceding lemma leads to elliptic curves with suitable twists.

\begin{lem}
\label{laatstestap}
Let $p$ be a prime, and suppose that $E$ is an elliptic curve over $\Q$ such that, for all $d \in \Q_p^{\ast}$, the topological group $E^d(\Q_p)$ is procyclic. Then $E$ has suitable twists with respect to $p$.
\end{lem}
\begin{proof}
Obviously, it suffices to assume $d=1$, and to show that there exists a twist $E^c$ of $E$ with $c \in \Q \cap \Q_p^{\ast 2}$ such that $E^c(\Q)$ is dense in $E^c(\Q_p)$.

Assume that $E$ is given by a Weierstrass curve $y^2=f(x)$ in $\p^2_\Q$. Let $(z,w)$ be a topological generator of $E(\Q_p)$. Let $(u,v) \in \A^2(\Q) \subset \p^2(\Q)$ be chosen sufficiently close to $(z,w)$, and such that both $f(u)$ and $v$ non-zero. Define $c = f(u)/v^2$. Since $c$ is arbitrarily close to $f(z)/w^2 = 1$, it is a $p$-adic square. Also, $(u,v)$ lies on the curve $cy^2=f(x)$, which we may identify with $E^c$. We claim that the multiples of $(u,v)$ lie dense in $E^c(\Q)$. Proof of claim: let $\alpha \in \Q_p^{\ast}$ be such that $\alpha^2 = c$. Note that $\alpha$ gets arbitrarily close to $-1$ or $1$. There is an isomorphism defined over $\Q_p$ given by:
\begin{align*}
\psi : E^c & \rightarrow E \\
(x,y) & \mapsto (x,\alpha y)
\end{align*}
Since $(u,v)$ was arbitrarily close to $(z,w)$, its image $(u,\alpha v)$ is arbitrarily close to $(z,\pm w)$, and both of these points are topological generators of $E(\Q_p)$. Since $\psi$ is a homeomorphism on $\Q_p$-points, $(u,v)$ is itself a topological generator of $E^c(\Q_p)$.
\end{proof}

\begin{prop}
\label{oneindigveel}
For any prime number $p$, there exist infinitely many elliptic curves $E$ over $\Q$ that have suitable twists with respect to $p$.
\end{prop}
\begin{proof}
This follows from Lemma \ref{existenceresult} and Lemma \ref{laatstestap}.
\end{proof}

\section{Proofs of Theorems \ref{swd} and \ref{cmtrick}}
\label{main}

At the end of this section we give the proofs of Theorems \ref{swd} and \ref{cmtrick}.
\begin{thm}
\label{big}
Let $p$ be a prime number and let $E$ be an elliptic curve over $\Q$ that has suitable twists with respect to $p$. Let $X = \Km(E \times E)$. Then $X(\Q)$ is dense in $X(\Q_p)$. 
\end{thm}
\begin{proof}
Recall that by $E_0$ we denote the complement of $E[2]$ in $E$. The inversion $-1$ on $E$ restricts to an involution of $E_0$, which we will also denote by $-1$. The quotient $(E_0 \times E_0)/\left\langle -1 \right\rangle$, where $-1$ acts diagonally, is a smooth subvariety $Y$ of $X$. Since no open neighborhood in $X(\Q_p)$ of a point in $X - Y$ can be contained in $X - Y$, it is enough to show that $Y(\Q)$ is dense in $Y(\Q_p)$. Observe that $Y$ may be identified with the open subset of
$$
z^2 = f(x)f(y),
$$
where $z$ is not equal to $0$.

Let $P = (\xi,\eta,\zeta)$ be a point of $Y(\Q_p)$. Let $d = f(\xi)$. By Definition \ref{geschikt}, there exists $c \in \Q^{\ast}$ such that $d/c \in \Q_p^{\ast}$ and $E^c(\Q)$ is dense in $E^c(\Q_p)$. We have a morphism:
\begin{align*}
q_c : E^c_0 \times E^c_0 & \rightarrow Y \\
(x_1,y_1),(x_2,y_2) & \mapsto (x_1,x_2,cy_1y_2)
\end{align*}
Furthermore, the point $P$ is the image under $q_c$ of the point $Q = ((\xi,1),(\eta,\zeta/f(\xi))) \in (E^c_0 \times E^c_0)(\Q_p)$. Since $E^c(\Q)$ is dense in $E^c(\Q_p)$, there exists a rational point $Q' \in (E^c_0 \times E^c_0)(\Q)$ such that $Q'$ is as close as we desire to $Q$, and hence such that $P' = q_c(Q') \in Y(\Q)$ is as close as we desire to $P$.
\end{proof}

\begin{rem}\upshape
What underlies our proof of Theorem \ref{big} is the fact that
$$
Y(\Q) = \coprod_{c} q_c((E^c_0 \times E^c_0)(\Q)),
$$
where the $q_c$ are as in the proof of Theorem \ref{big}, and where the $c$ are taken over a set of coset representatives of $\Q^{\ast}/\Q^{\ast 2}$ in $\Q^{\ast}$. The proof of Theorem \ref{big} relies on the existence, for any $P \in Y(\Q_p)$, of $c \in \Q^{\ast}$ such that (i) $P$ is in the image under $q_c$ of a $p$-adic point on $E^c_0 \times E^c_0$, and (ii) the rational points on $E^c_0 \times E^c_0$ lie $p$-adically dense. The existence of such a $c$ is precisely the condition that $E$ be suitable with respect to $p$.
\end{rem}

\begin{thm}
\label{j1728}
Let $E/\Q$ be the elliptic curve $y^2=x^3+x$ and let $X = \Km(E \times E)$. Then $X(\Q)$ is dense in $X(\Q_p)$ for all $p$ with $p \equiv 3 \pmod{4}$ and $p>7$.
\end{thm}
\begin{proof}
Let $p$ be a prime congruent to 3 mod 4. For $d \in \Q_p^{\ast}$, the twist $E^d$ of $E$ is given by the equation $y^2=x^3+d^2x$. By Lemma \ref{laatstestap} and Theorem \ref{big}, it suffices to show that $E^d(\Q_p)$  is procyclic for all $d \in \Q_p^{\ast}$. By changing to a $\Q_p$-isomorphic curve if necessary, it suffices to restrict to the case of $d \in \Q_p^{\ast}$ with $v_p(d)$ equal to $0$ or $1$.

First assume $v_p(d)=0$. Let $\wt{E}$ be the reduction of $E^d$ modulo $p$. Then $\# \wt{E}(\F_p) = p+1$. For $p>3$ this follows from the fact that $\wt{E}$ is supersingular \cite[V.4.5]{Silverman}; for $p=3$, one verifies it by a computation. We claim that $\wt{E}(\F_p)$ is cyclic. Suppose that $(\Z/\ell\Z)^2 \subset \wt{E}(\F_p)$ for some prime $\ell$. Then $p$ must split completely in $\Q(\zeta_\ell)$, giving $\ell \mid p-1$. On the other hand $\ell$ must certainly divide $\# \wt{E}(\F_p)=p+1$: therefore we must have $\ell = 2$. But since $x^3+d^2x$ splits into a linear and a quadratic irreducible polynomial over $\F_p$, we must have $\# \wt{E}(\F_p)[2] = 2$. This gives a contradiction, proving the claim.

By \cite[VII.2.1]{Silverman} and the fact that $E^d$ has good reduction at $p$, we have a short exact sequence:
$$
0 \rightarrow (E^d)^{(1)}(\Q_p) \rightarrow E^d(\Q_p) \rightarrow \wt{E}(\F_p) \rightarrow 0,
$$
where the kernel of reduction $(E^d)^{(1)}(\Q_p)$ of $E^d$ is isomorphic to $\Z_p$ by \cite[IV.6.4(b)]{Silverman}. We conclude that $E^d(\Q_p)$ is topologically isomorphic to the direct product of $\Z_p$ and a cyclic group of order $p+1$. Hence $E^d(\Q_p)$ is procyclic.

Now assume $v_p(d)=1$. Then $E^d$ has additive reduction with Kodaira type IV \cite[C.15]{Silverman}, hence we have a short exact sequence
$$
0 \rightarrow (E^d)^{(0)}(\Q_p) \rightarrow E^d(\Q_p) \rightarrow G \rightarrow 0,
$$ 
where $(E^d)^{(0)}(\Q_p)$ is topologically isomorphic to $\Z_p$ by Lemma \ref{swdlemma}, and $G$ is cyclic of order 1 or 3 (see \cite[C.15]{Silverman}). Again, $E^d(\Q_p)$ is topologically isomorphic to the direct product of $\Z_p$ and a cyclic group of order $1$ or $3$. Hence $E^d(\Q_p)$ is procyclic.
\end{proof}

\noindent
{\it Proof of Theorems \ref{swd} and \ref{cmtrick}.} Theorem \ref{swd} follows from Proposition \ref{oneindigveel} and Theorem \ref{big}. Theorem \ref{cmtrick} follows from Theorem \ref{j1728}.

\section{Acknowledgements}

It is a pleasure to thank Sir Peter Swinnerton-Dyer, Alexei Skorobogatov, Ronald van Luijk, Peter Stevenhagen and Jaap Top for useful discussions and encouragement. 

\bibliography{kummer-art-short}

\begin{thebibliography}{1}

\bibitem{Rene}
Ren\'e Pannekoek.
\newblock On $p$-torsion of $p$-adic elliptic curves with additive reduction,
  2012.
\newblock Preprint. See: http://arxiv.org/abs/1211.5833.

\bibitem{Silverman}
J.~Silverman.
\newblock {\em The Arithmetic of Elliptic Curves, Second Edition}.
\newblock Springer-Verlag, New York, 2009.

\bibitem{unpub}
H.P.F. Swinnerton-Dyer.
\newblock Density of rational points on certain surfaces.
\newblock Unpublished, 2010.

\end{thebibliography}

\end{document}